\documentclass[12pt]{amsart}

\usepackage{amsmath, amscd, amssymb,xypic, euler}
\usepackage[frame,cmtip,arrow,matrix,line,graph,curve]{xy}
\usepackage{graphpap, color, pstricks}
\usepackage[mathscr]{eucal}
\usepackage[pdftex,colorlinks,backref=page,citecolor=blue]{hyperref}

\newtheorem{theorem}{Theorem}[section]
\newtheorem{proposition}[theorem]{Proposition}

\newtheorem{lemma}[theorem]{Lemma}

\newtheorem{conjecture}[theorem]{Conjecture}

\theoremstyle{definition}
\newtheorem{definition}[theorem]{Definition}

\newtheorem{remark}[theorem]{Remark}

\newcommand{\PP}{\mathbb{P}}
\newcommand{\QQ}{\mathbb{Q}}

\newcommand{\cO}{\mathcal{O} }
\newcommand{\cA}{\mathcal{A} }

\newcommand{\proj}{\mathrm{Proj}\;}

\def\Mzn{\overline{M}_{0,n} }
\def\Mza{\overline{M}_{0,\cA} }

\def\Mzek{\overline{M}_{0,n\cdot \epsilon_k} }

\def\git{/\!/ }

\begin{document}

\title[Log canonical models for $\Mzn$]
{Log canonical models for the moduli space of pointed stable rational curves}
\date{January, 2011}
\author{Han-Bom Moon}
\address{Department of Mathematics, Seoul National University, Seoul 151-747, Korea}
\email{spring-1@snu.ac.kr}

\begin{abstract}
We run Mori's program for the moduli space of pointed stable rational 
curves with divisor $K +\sum a_{i}\psi_{i}$. 
We prove that, without assuming the F-conjecture,
the birational model for the pair is the Hassett's moduli space of 
weighted pointed stable rational curves, 
without any modification of weight coefficients.
\end{abstract}

\maketitle


\section{Introduction}\label{sec-introduction}
The Knudsen-Mumford space $\Mzn$, or the moduli space of pointed stable
rational curves is one of the most concrete and well-studied moduli spaces in 
algebraic geometry.
For example, it is well-known that $\Mzn$ is a smooth projective fine 
moduli space (\cite{Keel, Knudsen}). 
Also the cohomology ring, the Chow ring
and the Picard group are known (\cite{Keel}).
There are several concrete constructions by using 
explicit methods such as smooth blow-ups (\cite{Kapranov, Keel}) 
or by geometric invariant theory (GIT) as quotients by 
a small dimensional algebraic group (\cite{HuKeel, KiemMoon}). 
Futhermore, there are various different compactifications of 
the space of smooth pointed rational curves 
such as Hassett's moduli spaces of weighted 
pointed stable rational curves $\Mza$ (\cite{Hassett}), 
the GIT quotients of the product of the projective lines (\cite{Kapranov})
and the moduli spaces of pointed conics (\cite{GiansiracusaSimpson}).
All of these are birational models of $\Mzn$.

In spite of these numerous achievements, the birational geometric 
aspects of $\Mzn$ are not fully understood yet.
For instance, the Mori cone $\overline{NE}_{1}(\Mzn)$ (dually, the nef cone 
$Nef(\Mzn)$) is unknown. There is a conjectural description 
of this cone which is proved for $n \le 7$ (\cite{KeelMcKernan}).

\begin{conjecture}[F-conjecture]\label{cnj-F-conjecture}
Any effective curve in $\Mzn$ is numerically equivalent to a
nonnegative linear combination of vital curves. In other words,
every extremal ray of $\overline{NE}_{1}(\Mzn)$ is generated by 
vital curve classes.
\end{conjecture}

Recently, there has been a tremendous amount of interest in the birational 
geometry of $\Mzn$ (\cite{AlexeevSwinarski, Fedorchuk, FedorchukSmyth,
GiansiracusaSimpson, Hassett, HuKeel, Kapranov, Simpson})
and more generally of $\overline{M}_{g,n}$.
In particular, one can run Mori's program (or the minimal model program) 
for $\Mzn$ with a big $\QQ$-divisor $D$ of $\Mzn$, 
by finding a birational model
\begin{equation}\label{eqn-minimalmodel}
	\Mzn(D) := \proj \left(\bigoplus_{l \ge 0}H^{0}(\Mzn, 
	\cO(lD))\right)
\end{equation}
where the sum is taken over $l$ sufficiently 
divisible, and giving a moduli theoretic description to this birational model.

The most prominent two results in this direction are the following.
Set $m = \lfloor \frac{n}{2} \rfloor$. Let $\epsilon_{k}$ be a 
rational number in the range $\frac{1}{m+1-k} < \epsilon_{k} \le 
\frac{1}{m-k}$ for $k = 1, 2, \cdots, m-2$.
For $\epsilon > 0$, let $n\cdot \epsilon = (\epsilon, \cdots, \epsilon)$ 
be a \emph{symmetric} weight datum. 

\begin{theorem}[Simpson \cite{AlexeevSwinarski, 
FedorchukSmyth, KiemMoon, Simpson}]\label{thm-symmetriccase}
Let $\beta$ be a rational number satisfying $\frac{2}{n-1} < \beta
\le 1$ and let $D=\Mzn-M_{0,n}$ denote the total boundary divisor. 
Then the \emph{log canonical model} 
\begin{equation}\label{eqn-lcmodelsymmetric}
	\Mzn(K_{\Mzn}+\beta D) = \proj\left(\bigoplus_
	{l \ge 0} H^0(\Mzn, \cO( l (K_{\Mzn} +\beta D) ))\right)
\end{equation}
satisfies the following:
\begin{enumerate} 
	\item If $\frac{2}{m-k+2} < \beta \le \frac{2}{m-k+1}$ for $1\le k\le
	m-2$, then $\Mzn(K_{\Mzn}+ \beta D) \cong \Mzek$.
	\item  If $\frac{2}{n-1} < \beta \le \frac{2}{m+1}$, 
	then $\Mzn(K_{\Mzn}+\beta D) \cong (\PP^1)^n \git SL(2)$ 
	where the quotient is taken with respect to the symmetric linearization 
	$\cO(1,\cdots,1)$.
\end{enumerate}
\end{theorem}
The other result concerned \emph{non-symmetric} weights and higher genera
is the following theorem of Fedorchuk. It is an answer to the question of 
Hassett (\cite[Problem 7.1]{Hassett}).

\begin{theorem}\label{thm-fedorchuk}
\cite{Fedorchuk}
For every genus $g$ and weight datum $\cA$, there exists a 
\emph{log canonical divisor} 
$D_{g, \cA}$ on $\overline{M}_{g,n}$ such that the log canonical 
model $\overline{M}_{g, n}(K_{\overline{M}_{g,n}}+D_{g,\cA})$ is 
isomorphic to $\overline{M}_{g,\cA}$.
\end{theorem}

Very recently, in \cite{Moon1}, the author finds an \emph{universal} formula
generalizing Theorem \ref{thm-symmetriccase}
to non-symmetric weights $\cA = (a_{1},a_{2},\cdots, a_{n})$,
and proves it with assuming the F-conjecture.

\begin{theorem}\label{thm-mainthmintro}
Let $\cA = (a_{1},a_{2},\cdots, a_{n})$ be a weight datum.
Then the log canonical model 
$\overline{M}_{0,n}(K_{\Mzn}+ \sum_{i=1}^{n} a_{i}\psi_{i})$ 
is isomorphic to $\Mza$.
\end{theorem}

The aim of this paper is proving Theorem \ref{thm-mainthmintro}
\emph{without assuming the F-conjecture}.

\medskip

Here is an outline of the proof.
Let $\Delta_{\cA} = K_{\Mzn}+\sum_{i=1}^{n}a_{i}\psi_{i}$.
Let $\varphi_{\cA} : \Mzn \to \Mza$ be the reduction morphism. 
By computing the push-forwards and pull-backs of divisors 
(See Section \ref{sec-divisors}.),
we prove that 
$\Delta_{\cA}-\varphi_{\cA}^{*}\varphi_{\cA *}(\Delta_{\cA})$
is an effective divisor supported on the exceptional locus of $\varphi_{\cA}$.
Thus
\[
	H^{0}(\Mzn, \cO(\Delta_{\cA})) \cong 
	H^{0}(\Mza, \cO(\varphi_{\cA *}(\Delta_{\cA})))
\]
by \cite[Lemma 7.11]{Debarre}. 
Hence if we prove that $\varphi_{\cA *}(\Delta_{\cA})$ is ample on $\Mza$, 
then we have
\begin{eqnarray*}
	\Mzn(\Delta_{\cA}) &=& \proj \left(\bigoplus_{l \ge 0}
	H^{0}(\Mzn, \cO(l\Delta_{\cA}))\right)\\
	&\cong& \proj \left(\bigoplus_{l \ge 0}
	H^{0}(\Mza, \cO(l \varphi_{\cA *}(\Delta_{\cA}))) \right)
	\cong \Mza.
\end{eqnarray*}

For proving the ampleness of $\varphi_{\cA *}(\Delta_{\cA})$,
we follow the strategy of Fedorchuk in \cite{Fedorchuk}. 
Firstly, we can express $\varphi_{\cA *}(\Delta_{\cA})$
in terms of tautological divisors on $\Mza$.
Then by using a positivity result of Fedorchuk (Proposition \ref{prop-positivity})
and the induction on the dimension, 
we prove that $\varphi_{\cA *}(\Delta_{\cA})$ 
intersects non-negatively with all irreducible curves on $\Mza$, so is nef.
Moreover, we proved that small perturbations of 
$\varphi_{\cA *}(\Delta_{\cA})$ by boundary divisors are again nef.
Since the Neron-Severi vector space $N^{1}(\Mza)$ is generated by the 
boundary divisor classes, this implies that $\varphi_{\cA *}(\Delta_{\cA})$ 
lying on the interior of $Nef(\Mza)$, so it is ample by Kleiman's 
criterion.

\medskip

For boundary weight cases, Kapranov's morphism 
$\pi_{\cA} : \Mzn \to (\PP^{1})^{n}\git SL(2)$ 
plays the same role of the reduction morphism. 
By the similar strategy, we prove the following theorem in \cite{Moon1}.

\begin{theorem}\cite{Moon1}
\label{thm-mainthmboundarycaseintro}
Let $\cA = (a_{1}, \cdots, a_{n})$ be a boundary weight data, 
i.e., $\sum_{i=1}^{n} a_{i} = 2$. 
Then the log canonical model 
$\Mzn(K_{\Mzn}+\sum_{i=1}^{n}a_{i}\psi_{i})$ is
isomorphic to $(\PP^{1})^{n} \git_{L}  SL(2)$.
where $L$ is the linearization $\cO(a_{1},\cdots, a_{n})$.
\end{theorem}

By direct computation, it is easy to see that 
item (1) of Theorem \ref{thm-symmetriccase} 
is a special case of Theorem \ref{thm-mainthmintro} (see Remark 
\ref{rmk-symmetriccase}).
Also, we can regard item (2) of Theorem \ref{thm-symmetriccase} 
as a special case of Theorem \ref{thm-mainthmboundarycaseintro}.

It is well known that the dimension of the Neron-Severi vector space 
$N^{1}(\Mzn)$
is $2^{n-1}-{n \choose 2} -1$ (\cite{Keel}).
But by Theorem \ref{thm-mainthmintro}, to get Hassett's spaces $\Mza$,
it suffices to run Mori's program for an $n$-dimensional subcone of 
$N^{1}(\Mzn)$ only.
So we can guess that there are huge unknown families of birational models of 
$\Mzn$ other than Hassett's moduli spaces.

\medskip

This paper is organized as follows. In section \ref{sec-prelimilary},
we give some known facts about $\Mza$ and its divisor classes. 
Essentially there is no new result in this section. 
In section \ref{sec-proof}, we give a proof of Theorem 
\ref{thm-mainthmintro}.

\medskip

\textbf{Acknowledgement. }
It is a great pleasure to thank my advisor Young-Hoon Kiem.
Originally finding an universal formula for log canonical models of $\Mzn$ 
is a question raised by him. Without his patience and advice, it is impossible 
to finish this project.
I would also like to thank Maksym Fedorchuck 
for invaluable discussions and comments.


\section{Some Preliminaries}\label{sec-prelimilary}

\subsection{Moduli space of weighted pointed rational stable curves}
\label{sec-Hassettspace}
A \emph{weight datum} $\cA = (a_{1},a_{2},\cdots, a_{n})$ is a sequence of 
rational numbers such that $0 < a_{i} \le 1$ and $\sum_{i=1}^{n} a_{i} > 2$.
A family of rational curves with $n$ marked points over a base scheme $B$ 
consists of a flat proper morphism $\pi : C \to B$ 
whose geometric fibers are nodal connected rational curves, 
and $n$ sections $s_{1},s_{2},\cdots, s_{n}$ of $\pi$.

\begin{definition}\cite[Section 2]{Hassett}
A family of rational curves with $n$ marked points 
$\pi : (C, s_{1},\cdots,s_{n}) \to B$ is \emph{$\cA$-stable} if
\begin{enumerate}
	\item the sections $s_{1},\cdots, s_{n}$ lie in the smooth locus of $\pi$;
	\item for any subset $\{s_{i_{1}},\cdots, s_{i_{r}}\}$ of 
	nonempty intersection, $a_{i_{1}}+\cdots+a_{i_{r}}\le 1$;
	\item $\omega_{\pi}+\sum a_{i}s_{i}$ is $\pi$-ample.
\end{enumerate}
\end{definition}
For any weight data $\cA$, there exists a connected projective smooth 
moduli space $\Mza$ (\cite[Theorem 2.1]{Hassett}). 
Note that when $a_{1} = \cdots = a_{n} = 1$, $\Mza = \Mzn$.

Let $\cA = (a_{1}, \cdots, a_{n})$, 
$\mathcal{B} = (b_{1}, \cdots, b_{n})$ be two weight data and suppose that 
$a_{i} \ge b_{i}$ for all $i = 1, 2, \cdots, n$. 
Then there exists a birational \emph{reduction morphism}
(\cite[Theorem 4.1]{Hassett}) 
\[
	\varphi_{\cA,\mathcal{B}} : \Mza \to \overline{M}_{0, \mathcal{B}}.
\]
For $(C, s_{1},\cdots, s_{n}) \in \Mza$, $\varphi_{\cA, \mathcal{B}}
(C, s_{1},\cdots, s_{n})$ is obtained by collapsing components on which 
$\omega_{C} + \sum b_{i}s_{i}$ fails to be ample. Every reduction 
morphism is a divisorial contraction.

In this article, we use reduction morphisms from $\Mzn$ only. 
So we use more concise notation 
\[
	\varphi_{\cA} := \varphi_{(1, \cdots, 1), \cA} : \Mzn \to \Mza.
\]
It is a composition of blow-ups along smooth subvarieties 
(\cite{KiemMoon, Moon2}).


\subsection{Tautological divisors on $\Mza$}\label{sec-divisors}
In this section, we recall some information about several functorial divisors 
on $\Mza$. Most of results in this section are spread on 
many literatures for instance 
\cite{ArbarelloCornalba1, ArbarelloCornalba2, Fedorchuk, FedorchukSmyth,
HarrisMorrison, Hassett}. Some of results are simple generalization or 
modification of them.

Let $[n] = \{1,2,\cdots,n\}$. For $I \subset [n]$ such that 
$2 \le |I| \le n-2$, let $D_{I} \subset \Mzn$ be the closure of the locus of 
curves $C$ with two irreducible components $C_{I}, C_{I^{c}}$ such that 
$i$-th marked point lying on $C_{I}$ if and only if $i \in I$.
So $D_{I} = D_{I^{c}}$. These divisors are called \emph{boundary divisors}.
By \cite{Keel}, boundary divisors generate the Picard group 
$\mathrm{Pic}(\Mzn)$ and Neron-Severi vector space $N^{1}(\Mzn)$.

Let $\cA = (a_{1}, a_{2}, \cdots, a_{n})$ be a weight datum. 
For $I \subset [n]$, let $w_{I} := \sum_{i \in I}a_{i}$.
There are two kinds of boundary divisor classes in $\Mza$ for a general 
weight datum $\cA$.
\begin{enumerate}
	\item \emph{Boundary of nodal curves}: 
	Suppose that $w_{I^{c}} \ge w_{I} > 1$.
	Let $D_{I}$ be the divisor of $\Mza$ corresponding 
	the closure of the locus of curves with two irreducible components 
	$C_{I}, C_{I^{c}}$ and $s_{i} \in C_{I}$ if and only if $i \in I$.
	Let $D_{\mathrm{nod}}$ be the sum of all boundaries of 
	nodal curves.
	\item \emph{Boundary of curves with coincident sections}:
	Suppose that $I = \{i, j\}$ and $w_{I} \le 1$.
	Since $w = w_{[n]} > 2$, this implies $w_{I^{c}} > w_{I}$
	automatically.
	Let $D_{I}$ be the locus of $s_{i} = s_{j}$.
	Let $D_{\mathrm{sec}}$ be the sum of all boundaries of 
	curves with coincident sections.
\end{enumerate}
Since the reduction morphism $\varphi_{\cA}$ is a composition of 
smooth blow-ups, 
one can easily derive following push-forward and pull-back formulas 
for divisor classes.
  
\begin{lemma}\label{lem-pushforwardpullback}
Let $\varphi_{\cA} : \Mzn \to \Mza$ be the reduction morphism.
For $I \subset [n]$, let $w_{I} = \sum_{i \in I}a_{i}$.
Assume $w_{I} \le w_{I^{c}}$ for every $D_{I}$.
\begin{enumerate}
	\item $\varphi_{\cA *}(D_{I}) = \begin{cases}
	0, & |I| \ge 3 \mbox{ and } w_{i}\le 1\\
	D_{I},& \mbox{otherwise.}\end{cases}$
	\item $\varphi_{\cA}^{*}(D_{I}) = \begin{cases}
	D_{I} + \sum_{J \supset I, w_{J} \le 1}D_{J},
	& D_{I} \mbox{ is a boundary of curves with coincident sections}\\
	D_{I}, &\mbox{otherwise.}\end{cases}$
\end{enumerate}
\end{lemma}

Let $\pi : U \to \Mza$ be the universal curve and $\sigma_{i} : \Mza \to U$ 
for $i = 1, \cdots, n$ be the unversal sections.
Let $\omega = \omega_{U/\Mza}$ be the relative dualizing bundle.
Then we can define several tautological divisors on $\Mza$ by 
using $\omega$ and the intersection theory.
\begin{enumerate}
	\item The \emph{kappa class} is 
	$\kappa = \pi_{*}(c_{1}^{2}(\omega))$.
	This definition is different from $\kappa_{1}$ in \cite{ArbarelloCornalba2}.
	
	\item	For $1 \le i \le n$, let $\mathbb{L}_{i}$ 
	be the line bundle on $\Mza$,
	whose fiber over $(C, s_{1}, s_{2},\cdots, s_{n})$ is 
	$\Omega_C |_{s_{i}}$, 
	a cotangent space at $s_{i}$ in $C$.
	The $i$-th \emph{psi class} is
	$\psi_{i} = c_{1}(\mathbb{L}_{i})$.
	In terms of the intersection theory, 
	$\psi_{i} = \pi_{*}(\omega \cdot \sigma_{i}) = 
	\pi_{*}(-\sigma_{i}^{2})$.
	The \emph{total psi class} is $\psi = \sum_{i=1}^{n}\psi_{i}$.
	
	\item The boundary of curves with coincident sections $D_{\{i,j\}}$ 
	is equal to $\pi_{*}(\sigma_{i}\cdot \sigma_{j})$. 
\end{enumerate}
We focus on the genus zero case only, so the \emph{lambda class} 
$\lambda = c_{1}(\pi_{*}(\omega))$ is zero.

\medskip
Next, consider the push-forwards and pull-backs of several divisors.

\begin{lemma}\label{lem-pushforwardpullbackpsi}
Let $\varphi_{\cA} : \Mzn \to \Mza$ be the reduction morphism.
\begin{enumerate}
	\item $\varphi_{\cA *}(K_{\Mzn}) = K_{\Mza}$.
	\item $\varphi_{\cA *}(\psi_{i}) = \psi_{i} + 
	\sum_{\substack{j \ne i\\ a_{i}+a_{j} \le 1}}D_{\{i, j\}}$.
	\item $\varphi_{\cA}^{*}(\psi_{i}) = 
	\psi_{i} - \sum_{\substack{i \in I\\ w_{I} \le 1}}D_{I}$.
\end{enumerate}
\end{lemma}
\begin{proof}
Since the discrepancy is supported on the exceptional locus, 
item (1) follows immediately.
Item (2) and (3) are more careful observations of the proof of 
\cite[Lemma 2.4]{FedorchukSmyth} and 
\cite[Lemma 2.8]{FedorchukSmyth} respectively.
Item (3) is also a corollary of cumbersome computation 
using Lemma \ref{lem-pushforwardpullback} and 
\cite[Lemma 3.1]{FarkasGibney}.
\end{proof}

For $I = \{i_{1}, \cdots, i_{r}\}\subset [n]$,
let $D_{I}$ be a boundary of nodal curves.
Set $I^{c} = \{j_{1}, \cdots, j_{s}\}$.
Then $D_{I}$ is isomorphic to $\overline{M}_{0,\cA_{I}} \times 
\overline{M}_{0,\cA_{I^{c}}}$
where $\cA_{I} = (a_{i_{1}}, \cdots, a_{i_{r}}, 1)$
and $\cA_{I^{c}} = (a_{j_{1}}, \cdots, a_{j_{s}},1)$.
Let $\eta_{I} : \overline{M}_{0,\cA_{I}} \times \overline{M}_{0, \cA_{J}}
\to D_{I} \hookrightarrow \Mza$ be the inclusion morphism.
Define $\pi_{i}$ for $i=1,2$ as the projection from 
$\overline{M}_{0,\cA_{I}} \times \overline{M}_{0,\cA_{J}}$
to the $i$-th component.

\begin{lemma}\label{lem-restrictiontoboundary}
Let $\eta_{I} : \overline{M}_{0,\cA_{I}} \times \overline{M}_{0, \cA_{J}}
\to D_{I} \hookrightarrow \Mza$ be the inclusion morphism.
Let $p$ (resp. $q$) denote the last index of $\cA_{I}$ (resp. $\cA_{J}$) 
with weight one.
\begin{enumerate}
	\item $\eta_{I}^{*}(\kappa) = \pi_{1}^{*}(\kappa + \psi_{p})+
	\pi_{2}^{*}(\kappa + \psi_{q})$.
	\item $\eta_{I}^{*}(\psi_{i}) = \begin{cases}\pi_{1}^{*}(\psi_{i}),&
	i \in I\\ \pi_{2}^{*}(\psi_{i}),& i \in I^{c}.\end{cases}$
	\item For $J \subset [n]$, suppose that $D_{J}$ be a boundary 
	of nodal curves.\\
	$\eta_{I}^{*}(D_{J}) = \begin{cases}
	\pi_{1}^{*}(D_{J}),& J \subsetneq I\\
	\pi_{2}^{*}(D_{J}),& J \subsetneq I^{c}\\
	\pi_{1}^{*}(\psi_{p})+\pi_{2}^{*}(\psi_{q}),& J=I\\
	0, & \mbox{otherwise}.\end{cases}$
	\item Suppose that $a_{i}+a_{j} \le 1$.\\
	$\eta_{I}^{*}(D_{\{i,j\}}) = \begin{cases}
	\pi_{1}^{*}(D_{\{i,j\}}),& i, j \in I\\
	\pi_{2}^{*}(D_{\{i,j\}}),& i, j \in I^{c}\\
	0,& \mbox{otherwise}.\end{cases}$
\end{enumerate}
\end{lemma}
\begin{proof}
The proof of these items are essentially identical to the case of $\Mzn$.
Item (1) is in \cite[Section 1]{ArbarelloCornalba1}.
Items (2), (4) are clear.
The only non obvious part of item (3) is due to 
\cite[Proposition 3.31]{HarrisMorrison}.
\end{proof}

Let $J \subset [n]$ be a \emph{maximal} subset of $[n]$ such that 
$\sum_{j \in J}a_{j} \le 1$.
Let $\cA'$ be the new weight data obtained by replacing weights  
indexed by $J$ by one weight $\sum_{j \in J}a_{j}$.
Then the locus of $\sigma_{i}= \sigma_{j}$ for all $i, j \in J$ 
is isomorphic to $\overline{M}_{0, \cA'}$ because we can replace 
sections $\{\sigma_{j}\}_{j \in J}$ by one section with weight 
$\sum_{j \in J}a_{j}$.
Let $\chi_{J} : \overline{M}_{0, \cA'} \to \Mza$ be 
the replacement morphism.

\begin{lemma}\label{lem-restrictiontosectionalboundary}
Let $\chi_{J} : \overline{M}_{0,\cA'} \to \Mza$ be the replace morphism.
Let $p$ denote the unique index of $\cA'$ replacing indices in $J$.
\begin{enumerate}
	\item $\chi_{J}^{*}(\psi_{i}) = \begin{cases}
	\psi_{i},&i \notin J \\ \psi_{p},& i \in J.\end{cases}$
	\item $\chi_{J}^{*}(D_{\mathrm{nod}})
	= D_{\mathrm{nod}}$.
	\item Suppose that $D_{\{i,j\}}$ is a boundary of curves with 
	coincident sections.\\
	$\chi_{J}^{*}(D_{\{i,j\}}) = \begin{cases}
	D_{\{i,j\}},& i, j \notin J\\
	D_{\{i,p\}}, & i \notin J, j \in J\\
	-\psi_{p},& i,j \in J.\end{cases}$
\end{enumerate}
\end{lemma}
\begin{proof}
Essentially this is a restatement of \cite[Lemma 2.9]{FedorchukSmyth}.
\end{proof}

Finally, let us recall the canonical divisor of $\Mza$.
The follwing formula is a consequence of Hassett's computation of 
the canonical divisor and the weighted version of Mumford's relation 
$\kappa = -D_{\mathrm{nod}}$
(\cite[(3.15)]{ArbarelloCornalba2}).
\begin{lemma}\label{lem-canonicaldivisorMza}
\begin{equation}\label{eqn-canonicaldivisorMza}
	K_{\Mza} = -2D_{\mathrm{nod}}+\sum_{i=1}^{n}\psi_{i}
	= 2\kappa + \sum_{i=1}^{n}\psi_{i}.
\end{equation}
\end{lemma}
\begin{proof}
By \cite[Section 3.3.1]{Hassett},
\begin{equation}\label{eqn-canonicaldivisor}
	K_{\Mza} = \frac{13}{12}\kappa - \frac{11}{12}D_{\mathrm{nod}}
	+\sum_{i=1}^{n}\psi_{i}.
\end{equation}
By the ordinary Mumford's relation $\kappa = - D_{\mathrm{nod}}$ 
on $\Mzn$, $\varphi_{\cA *}(D_{\mathrm{nod}}) = 
D_{\mathrm{nod}}+D_{\mathrm{sec}}$
and Lemma \ref{lem-pushforwardpullbackpsi}, it is straightforward to check
\[
	K_{\Mza} = \varphi_{\cA *}(K_{\Mzn})
	= \varphi_{\cA *}(\frac{13}{12}\kappa - 
	\frac{11}{12}D_{\mathrm{nod}} + \sum_{i=1}^{n}\psi_{i})
	= -2D_{\mathrm{nod}}+\sum_{i=1}^{n}\psi_{i}.
\]
Thus $\kappa = - D_{\mathrm{nod}}$ on $\Mza$ too.
Substitute it to the right side of (\ref{eqn-canonicaldivisor}), 
we get the above formula.
\end{proof}

\section{Proof of Theorems}\label{sec-proof}
In this section, we prove our main theorem.
Through this section, we will assume $n \ge 4$.
If $n = 3$, then $\overline{M}_{0,3}$ is a point,
so there is nothing to prove.
 
\begin{theorem}\label{thm-mainthm}
Let $\cA = (a_{1},a_{2},\cdots, a_{n})$ be a weight datum.
Then the log canonical model 
$\overline{M}_{0,n}(K_{\Mzn}+ \sum_{i=1}^{n} a_{i}\psi_{i})$ 
is isomorphic to $\Mza$.
\end{theorem}

\begin{proof}
Fix a weight datum $\cA = (a_{1},a_{2},\cdots, a_{n})$.
Let $\Delta_{\cA} = K_{\Mzn}+\sum_{i=1}^{n}a_{i}\psi_{i}$.
Set $C = \{I \subset [n] | w_{I}= \sum_{i \in I}a_{i}\le 1, 
2 \le |I| \le n-2 \}$.
By Lemma \ref{lem-pushforwardpullbackpsi} and \ref{lem-canonicaldivisorMza},
it is straightforward to check that 
\begin{equation}\label{eqn-pushforward}
\begin{split}
	\varphi_{\cA *}(\Delta_{\cA}) &=
	K_{\Mza} + \sum_{i=1}^{n}a_{i}\psi_{i} + 
	\sum_{\substack{i<j\\a_{i}+a_{j}\le 1}}(a_{i}+a_{j})D_{\{i,j\}}\\
	&= -2D_{\mathrm{nod}} + \sum_{i=1}^{n}(1+a_{i})\psi_{i}
	+ \sum_{\substack{i<j\\a_{i}+a_{j}\le 1}}(a_{i}+a_{j})D_{\{i,j\}}.
\end{split}
\end{equation}
By Lemma \ref{lem-pushforwardpullback} and 
\ref{lem-pushforwardpullbackpsi},
\begin{equation}\label{eqn-pushforwardpullback}
\begin{split}
	\varphi_{\cA}^{*}\varphi_{\cA *}(\Delta_{\cA}) &=
	-2D_{\mathrm{nod}}+2 \sum_{I \in C}D_{I}+\sum_{i=1}^{n}
	(1+a_{i})\psi_{i} - \sum_{I \in C}(|I|+w_{I})D_{I}\\
	&+ \sum_{\substack{i<j \\ a_{i}+a_{j}}}(a_{i}+a_{j})D_{\{i,j\}}
	+ \sum_{\substack{I \in C\\ |I| \ge 3}}(|I|-1)w_{I}D_{I}\\
	&= -2D_{\mathrm{nod}}+\sum_{i=1}^{n}(1+a_{i})\psi_{i} 
	+ \sum_{I \in C}(|I|-2)(w_{I} -1)D_{I}.
\end{split}
\end{equation}
So 
\begin{equation}\label{eqn-difference}
	\Delta_{\cA} - \varphi_{\cA}^{*}\varphi_{\cA *}(\Delta_{\cA})
	= \sum_{I \in C}(|I|-2)(1-w_{I})D_{I}.
\end{equation}
Note that for every $I \in C$, 
$|I| \ge 2$ and $w_{I} \le 1$ by the definition of $C$.
So the difference $\Delta_{\cA} - 
\varphi_{\cA}^{*}\varphi_{\cA *}(\Delta_{\cA})$
is supported on the exceptional locus of $\varphi_{\cA}$ 
and effective.
This implies that 
\[
	H^{0}(\Mzn, \Delta_{\cA}) \cong 
	H^{0}(\Mzn, \varphi_{\cA}^{*}\varphi_{\cA *}(\Delta_{\cA})) \cong
	H^{0}(\Mza, \varphi_{\cA *}(\Delta_{\cA}))
\]
by \cite[Lemma 7.11]{Debarre}.
The same statement holds for a positive multiple of 
$\Delta_{\cA}$, too.
Therefore from the definition of the log canonical model, we get
\begin{equation}\label{eqn-lcmodel}
\begin{split}
	\Mzn(K_{\Mzn}+\sum_{i=1}^{n}a_{i}\psi_{i})
	&=  \proj\left(\bigoplus_{l \ge 0} 
	H^0(\Mzn, \cO(l \Delta_{\cA}))\right)\\
	&= \proj \left(\bigoplus_{l \ge 0}
	H^{0}(\Mza, \cO(l \varphi_{\cA *}(\Delta_{\cA})))\right).
\end{split}
\end{equation}
If we prove $\varphi_{\cA *}(\Delta_{\cA})$ is ample, 
then the last birational model is exactly $\Mza$.
So to prove the main theorem, it suffices to show that 
$\varphi_{\cA *}(\Delta_{\cA})$ is ample on $\Mza$.
This is done in Proposition \ref{prop-nef} and \ref{prop-ample}.
\end{proof}

\begin{proposition}\label{prop-nef}
Let $\cA = (a_{1}, \cdots, a_{n})$ be a weight datum
and let $\Delta_{\cA} = K_{\Mzn}+\sum_{i=1}^{n}a_{i}\psi_{i}$.
Then for the reduction morphism $\varphi_{\cA} : \Mzn \to \Mza$, 
$\varphi_{\cA *}(\Delta_{\cA})$ is a nef divisor on $\Mza$.
\end{proposition}

The key ingredient is the following positivity result of Fedorchuk \cite{Fedorchuk}.
Fedorchuk gives an elementary and beautiful intersection theoretical proof of 
this result. As Fedorchuk mentioned in \cite{Fedorchuk},
it can be proved by using the semipositivity method of Koll\'ar in
\cite[Corollary 4.6, Proposition 4.7]{Kollar}.

\begin{proposition}\cite[Proposition 2.1]{Fedorchuk}\label{prop-positivity}
Let $\pi : S \to B$ be a generically smooth family of nodal curves of arithmetic 
genus $g$, with $n$ sections $\sigma_{1}, \cdots, \sigma_{n}$
over a smooth complete curve $B$. 
For a weight datum $\cA = (a_{1}, \cdots, a_{n})$, suppose that 
\[
	L := \omega_{\pi}+\sum_{i=1}^{n} a_{i}\sigma_{i}
\]
is $\pi$-nef. Suppose further that $\sigma_{i_{1}}, \cdots, \sigma_{i_{k}}$
can coincide only if $\sum_{j} \sigma_{i_{j}} \le 1$. 
Then $L$ is nef on $S$.
\end{proposition}
If $\pi : S \to B$ is a generically smooth family of $\cA$-stable curves,
or more generally $\cA$-semi-stable curves (allowing irreducible components 
with $2$ nodes and no marked points),
then the assumptions of Proposition \ref{prop-positivity} is satisfied by 
the definition of $\cA$-stablility.

We need an effectivity result first.
\begin{lemma}\label{lem-effectivity}
Let $\pi : S \to B$ be a family of $\cA$-stable curves with $n$ sections 
$\sigma_{1}, \cdots, \sigma_{n}$ over a smooth complete curve $B$.
Then $2\omega_{\pi} + \sum_{i=1}^{n} \sigma_{i}$ is effective.
\end{lemma}
\begin{proof}
We will use induction on $n$. For $n=4$ case, the result is a direct computation.

By \cite[118p]{HarrisMorrison}, $S$ has at worst $A_{k}$ singularities only.
An $A_{k}$ singularity is Du Val, so if $\rho : \tilde{S} \to S$ is a minimal 
resolution, then $\omega_{\pi \circ \rho} = \rho^{*}(\omega_{\pi})$ and 
$\rho_{*}(\omega_{\pi \circ \rho}) = \omega_{\rho}$. Thus 
we may assume that $S$ is smooth.

Suppose that for $J \subset [n]$ with $|J| \ge 2$, $\sigma_{i} = \sigma_{j}$
for all $i, j \in J$.
We may assume that $J = \{1,2,\cdots, m\}$ for some $m \le n$.
Then by pull-back along $\chi_{J} : \overline{M}_{0,\cA'} \to \Mza$ 
(see Section \ref{sec-divisors}), 
we may assume that $(\pi : S \to B, \sigma_{1}, \cdots, \sigma_{n})$ is 
a family of  $\cA'$-stable curves 
$(\pi : S \to B, \sigma_{m}, \sigma_{m+1}, \cdots, \sigma_{n})$
with $|J| - 1$ additional sections 
$\sigma_{1}, \sigma_{2}, \cdots, \sigma_{m-1}$. 
By the induction hypothesis, $2\omega_{\pi} + \sum_{i = m}^{n}\sigma_{i}$ 
is effective.
So $2\omega_{\pi}+\sum_{i=1}^{n}\sigma_{i}
= (2\omega_{\pi}+\sum_{i=m}^{n}\sigma_{i})
+\sum_{i=1}^{m-1}\sigma_{i}$
is effective, too.
Thus we may assume that all sections are distinct.

After taking several blow-ups along points with two or more sections meet,
we get a family of $(1,1, \cdots, 1)$-semi-stable curves 
$(\pi_{1} : S_{1} \to B, \sigma_{1}^{1}, \cdots, \sigma_{n}^{1}$).
Let $\rho_{1} : S_{1} \to S$ be the blow-up.
If there exist $(-1)$ curves with exactly $2$ sections, 
after contracting these $(-1)$ curves by blowing-down, we get a family 
$(\pi_{2} : S_{2} \to B, \sigma_{1}^{2}, \cdots, \sigma_{n}^{2})$
of $(1/2, \cdots, 1/2)$-semi-stable curves.
Let $\rho_{2} : S_{1} \to S_{2}$ be the blow-down morphism.
Over $S_{2}$, $2\omega_{\pi_{2}} + \sum_{i=1}^{n}\sigma_{i}^{2}$ 
is nef by Proposition \ref{prop-positivity} and thus effective.

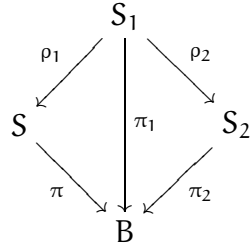
\begin{equation}\label{eqn-blowupdiagram}
	\xymatrix{&S_{1}\ar[ld]_{\rho_{1}}\ar[rd]^{\rho_{2}}
	\ar[dd]^{\pi_{1}}&\\
	S\ar[rd]_{\pi}&& S_{2}\ar[ld]^{\pi_{2}}\\
	& B}
\end{equation}

From $(1/2, \cdots, 1/2)$-stability, we know that 
for each point in $S_{2}$, at most two sections meet at that point.
Let $x_{1}, \cdots, x_{k}$ be points with coincident sections.
Then $\rho_{2}$ is the blow-up along $x_{1}, \cdots, x_{k}$.
Let $E_{1}, \cdots, E_{k}$ be the exceptional divisors.
By the blow-up formula, 
$\omega_{\pi_{1}} = \rho_{2}^{*}(\omega_{\pi_{2}})
+ \sum_{j=1}^{k}E_{j}$.
Also $\sum_{i=1}^{n} \sigma_{i}^{1} = \rho_{2}^{*}(\sum_{i=1}^{n}
\sigma_{i}^{2})-2\sum_{j=1}^{k}E_{j}$.
Thus
\begin{equation}
	2\omega_{\pi_{1}}+\sum_{i=1}^{n}\sigma_{i}^{1}
	= \rho_{2}^{*}(2\omega_{\pi_{2}}+\sum_{i=1}^{n}\sigma_{i}^{2}),
\end{equation}
so $2\omega_{\pi_{1}}+\sum_{i=1}^{n}\sigma_{i}^{1}$ is effective.

Finally, $\rho_{1 *}(\omega_{\pi_{1}}) = \omega_{\pi}$ and 
$\rho_{1 *}(\sigma_{i}^{1}) = \sigma_{i}$ since $\rho_{1}$ is a composition
of point blow-ups. 
Thus $2 \omega_{\pi}+\sum_{i=1}^{n}\sigma_{i} = 
\rho_{1 *}(2\omega_{\pi_{1}}+\sum_{i=1}^{n}\sigma_{i}^{1})$ is 
a push-forward of an effective divisor. Hence it is effective, too.
\end{proof}

\begin{proof}[Proof of Proposition \ref{prop-nef}]
For $n = 4$ case, since $\Mza \cong \Mzn \cong \PP^{1}$, 
the result is a consequence of a simple direct computation.
So we can use the induction on the dimension $n$.

To prove the nefness of $\varphi_{\cA *}(\Delta_{\cA})$, 
it suffices to show that for every complete irreducible curve $B \to \Mza$, 
the restriction of $\varphi_{\cA *}(\Delta_{\cA})|_{B}$ has 
nonnegative degree.
By composing the normalization $B^{\nu} \to B$, 
we may assume that $B$ is smooth.

By equation \eqref{eqn-pushforward} and $\kappa = - D_{\mathrm{nod}}$
(see the proof of Lemma \ref{lem-canonicaldivisorMza}),
it is straightforward to check that 
\begin{align}
	\varphi_{\cA *}(\Delta_{\cA}) &=
	2\kappa + \sum_{i=1}^{n}(1+a_{i})\psi_{i}
	+ \sum_{\substack{i<j\\a_{i}+a_{j}\le 1}}(a_{i}+a_{j})D_{\{i,j\}}
	\label{eqn-pushforwardDelta}\\
	&= \pi_{*}\Big(
	2\omega^{2}+\sum_{i=1}^{n}(1+a_{i})(\omega \cdot \sigma_{i})
	+\sum_{\substack{i < j\\a_{i}+a_{j}\le 1}}(a_{i}+a_{j})
	(\sigma_{i}\cdot \sigma_{j})\Big)
	\label{eqn-pushforwardDelta2}.
\end{align}
For a boundary divisor $D_{I}$ of nodal curves, 
let $\eta_{I} :\overline{M}_{0,\cA_{I}} \times 
\overline{M}_{0,\cA_{I^{c}}} \to D_{I}
\hookrightarrow \Mza$ be the inclusion of boundary.
We will use the same notation in Section \ref{sec-divisors}.
By Lemma \ref{lem-restrictiontoboundary} and \eqref{eqn-pushforwardDelta}, 
it is straightforward to check 
\begin{equation}\label{eqn-pullbackboundary}
	\eta_{I}^{*}(\varphi_{\cA *}(\Delta_{\cA})) 
	= \pi_{1}^{*}(\varphi_{\cA_{I} *}(\Delta_{\cA_{I}}))
	+ \pi_{2}^{*}(\varphi_{\cA_{I^{c}} *}(\Delta_{\cA_{I^{c}}})).
\end{equation}
Thus for a curve $B$ supported on a boundary, the degree of 
$\varphi_{\cA *}(\Delta_{\cA})$ is non-negative by induction.
Therefore it suffices to check for a family $S \to B$ whose general fiber 
is a nonsingular curve.

Note that $\omega \cdot \sigma_{i} = -\sigma_{i}^{2}$ by adjunction formula
and $\sigma_{i}\cdot \sigma_{j} = 0$ if $a_{i}+a_{j} > 1$.
Therefore 
\begin{equation}\label{eqn-neftimeseffective}
\begin{split}
	&\;
	2\omega^{2}+\sum_{i=1}^{n}(1+a_{i})(\omega \cdot \sigma_{i})
	+\sum_{\substack{i < j\\a_{i}+a_{j}\le 1}}(a_{i}+a_{j})
	(\sigma_{i}\cdot\sigma_{j})\\
	=&\; 2\omega^{2}+\sum_{i=1}^{n} (\omega \cdot \sigma_{i}) 
	+ \sum_{i=1}^{n} 2 a_{i}(\omega \cdot \sigma_{i})
	+\sum_{i=1}^{n} a_{i}\sigma_{i}^{2}+\sum_{i < j}(a_{i}+a_{j})
	(\sigma_{i}\cdot	\sigma_{j})\\
	=&\; (\omega+\sum_{i=1}^{n} a_{i}\sigma_{i})
	\cdot(2\omega + \sum_{i=1}^{n} \sigma_{i}).
\end{split}
\end{equation}
Hence it suffices to check that 
$\deg \pi_{*}\big((\omega+\sum_{i=1}^{n} a_{i}\sigma_{i})
\cdot(2\omega + \sum_{i=1}^{n} \sigma_{i})\big)|_{B} \ge 0$.
By Proposition \ref{prop-positivity}, 
$\omega+\sum_{i=1}^{n} a_{i}\sigma_{i}$ is nef on $S$.
By Lemma \ref{lem-effectivity},
$2\omega + \sum_{i=1}^{n} \sigma_{i}$ is effective on $S$.
Thus the intersection is non-negative and the result follows.
\end{proof}

Next, we prove the ampleness of $\varphi_{\cA *}(\Delta_{\cA})$.
This is an application of the perturbation technique of 
Fedorchuk and Smyth introduced in \cite{FedorchukSmyth}.
\begin{proposition}\label{prop-ample}
Within the same assumption of Proposition \ref{prop-nef}, 
$\varphi_{\cA *}(\Delta_{\cA})$ is an ample divisor on $\Mza$.
\end{proposition}

\begin{proof}
We will prove that the following statement:
For $\Mza$, there exists $\epsilon_{\cA} > 0$ such that 
$\varphi_{\cA *}(\Delta_{\cA})\cdot B \ge \epsilon_{\cA}$
for every irreducible curve class $B$.
This implies that $\varphi_{\cA *}(\Delta_{\cA})$ lies on the 
interior of $Nef(\Mza)$, so by Kleiman's criterion, 
$\varphi_{\cA *}(\Delta_{\cA})$ is ample.

We will use the induction on $n$.
When $n=4$, then $\Mza \cong \PP^{1}$ and the result is straightforward.

Let $B$ be an integral complete curve on $\Mza$.
Since we only consider the intersection numbers only,
we may assume $B$ is nonsingular by applying normalization.
We will divide into three cases:

\medskip
(1) $B$ is in a component of nodal boundary.

\smallskip
By (\ref{eqn-pullbackboundary}) and the induction hypothesis, when we restrict 
$\varphi_{\cA *}(\Delta_{\cA})$ to a component of boundary of nodal curves, 
the restriction is ample and there is a lower bound of intersection numbers.

\medskip
(2) A general point of $B$ parameterizes smooth curve and 
there exist $J \subset [n]$ with $|J| \ge 2$ such that 
$\sigma_{i} = \sigma_{j}$ for all $i,j \in J$ and $\sigma_{i}^{2} < 0$.

\smallskip
We may assume that $J$ is maximal among such subsets.
In this case, $B$ is contained in the image of 
$\chi_{J} : \overline{M}_{0, \cA'} \to \Mza$ defined in 
section \ref{sec-divisors}.
Let $p$ be the unique index of $\cA'$ replacing indices in $J$.
Then by \eqref{eqn-pushforward} and Lemma 
\ref{lem-restrictiontosectionalboundary},
\begin{equation}\label{eqn-replacementpullback}
\begin{split}
	&\;\chi_{J}^{*}(\varphi_{\cA *}(\Delta_{\cA}))
	= \chi_{J}^{*}\Big(-2 D_{\mathrm{nod}} + 
	\sum_{i=1}^{n}(1+a_{i})\psi_{i} + 
	\sum_{\substack{i < j\\ a_{i}+a_{j}\le 1}}(a_{i}+a_{j})D_{\{i,j\}}\Big)\\
	&=\;
	-2 D_{\mathrm{nod}} + \sum_{i \in J^{c}}(1+a_{i})\psi_{i}
	+ \sum_{i \in J}(1+a_{i})\psi_{p}
	+ \sum_{i<j,\;i, j \in J^{c}}(a_{i}+a_{j})D_{\{i,j\}}\\
	&\;+ \sum_{i \in J,\; j \in J^{c}}(a_{i}+a_{j})D_{\{p,j\}}
	- \sum_{i<j,\;i,j \in J}(a_{i}+a_{j}) \psi_{p}\\
	&=\;
	-2 D_{\mathrm{nod}} + \sum_{i \in J^{c}}(1+a_{i})\psi_{i}
	+ (1+\sum_{j \in J}a_{j})\psi_{p} + (|J|-1)\psi_{p}
	+ \sum_{i<j,\;i,j \in J^{c}}(a_{i}+a_{j})D_{\{i,j\}}\\
	&+ \sum_{j \in J^{c}}\Big((\sum_{i \in J}a_{i})+a_{j}\Big)D_{\{p,j\}}
	+ (|J|-1)\sum_{j \in J^{c}}a_{j}D_{\{p,j\}}
	-(|J|-1)\Big(\sum_{i \in J}a_{i}\Big)\psi_{p}\\
	&=\; \varphi_{\cA' *}(\Delta_{\cA'})
	+(|J|-1)\Big((1-\sum_{i \in J}a_{i})\psi_{p}
	+\sum_{j \in J^{c}}a_{j}D_{\{p,j\}}\Big).
\end{split}
\end{equation}

By induction hypothesis, $\varphi_{\cA' *}(\Delta_{\cA'})\cdot B$ 
is bounded below by a positive number. 
Since $\sigma_{p}^{2} < 0$ by assumption, 
$\psi_{p} = -\sigma_{p}^{2} > 0$ on $B$.
The last summand is nonnegative since $D_{\{p,i\}}$ does not cover 
whole $B$ because of the maximality of $J$.
Hence there exists a positive lower bound of the intersection number 
$\varphi_{\cA *}(\Delta_{\cA})\cdot B$.

\begin{remark}\label{rmk-errorterm}
Indeed, the term $(1-\sum_{i \in J}a_{i})\psi_{p}+
\sum_{j \in J^{c}}a_{j}D_{\{p.j\}}$ 
in the last row of \eqref{eqn-replacementpullback} 
is $C_{p}$ in \cite[Theorem 1]{Fedorchuk}.
Fedorchuk proved that $C_{p}$ is nef on $\Mza$.
\end{remark}

(3) Otherwise.

\smallskip
In this case, a general point of $B$ parameterizes a smooth curve.
Note that there exists $\delta > 0$ such that 
every $\cA = (a_{1}, a_{2}, \cdots, a_{n})$-stable curve is also 
$\cA_{1} = (a_{1}-\delta, a_{2}-\delta, \cdots, 
a_{n}-\delta)$-stable too.
Therefore $\Mza = \overline{M}_{0, \cA_{1}}$
and $\varphi_{\cA} = \varphi_{\cA_{1}}$,
so $\varphi_{\cA *}(\Delta_{\cA_{1}})$ is nef by 
Proposition \ref{prop-nef}.
Thus
\begin{eqnarray*}
	\varphi_{\cA *}(\Delta_{\cA}) &=& 
	\varphi_{\cA *}(K_{\Mzn}+\sum_{i=1}^{n} a_{i}\psi_{i})\\
	&=& 
	\varphi_{\cA *}(K_{\Mzn}+\sum_{i=1}^{n} (a_{i}-\delta)\psi_{i})
	+ \delta \varphi_{\cA *}(\psi)
	= \varphi_{\cA *}(\Delta_{\cA_{1}}) + \delta\varphi_{\cA *}(\psi).
\end{eqnarray*}
On $\Mzn$, $\psi = \sum_{j=1}^{\lfloor n/2 \rfloor}
\frac{j(n-j)}{n-1}D_{j}$ by \cite[Lemma 1]{FarkasGibney}.
Thus $\varphi_{\cA *}(\Delta_{\cA})$ is a sum of a nef divisor 
$\varphi_{\cA *}(\Delta_{\cA_{1}})$ and an effective divisor 
$\delta \varphi_{\cA}(\psi) + P$ supported on the boundary.
Note that $\psi$ and $\varphi_{\cA *}(\psi)$ are \emph{positive} 
linear combinations of \emph{all} boundary components.
So if we take a divisor $P$ with small coefficients
with respect to boundary divisors, then
$\delta\varphi_{\cA *}(\psi) + P$ is an effective sum of boundary divisors.

We claim that $(\varphi_{\cA *}(\Delta_{\cA})+P) \cdot B$ is nonnegative.
Since $B$ does not lie on nodal boundary divisors,
$D_{I} \cdot B \ge 0$ for all boundary divisors of nodal curves.
Also, since $B$ has no coincident sections with negative self intersections, 
$D_{I} \cdot B \ge 0$ for all divisors of coincident sections, too.
Therefore $B$ intersects non-negatively 
with $\varphi_{\cA *}(\Delta_{\cA})+P$.
So for any metric $|| \cdot ||$ on $N^{1}(\Mza)$, 
there exists $\alpha > 0$ such that if $|| P || < \alpha$, then 
$(\varphi_{\cA *}(\Delta_{\cA})+P) \cdot B \ge 0$ for every irreducible 
curve $B$.
Note that $\alpha$ depends on only $\delta$, but not on $B$.
Thus there exists a positive lower bound of the intersection number 
$\varphi_{\cA *}(\Delta_{\cA})\cdot B$.

Note that there exists only \emph{finitely many} strata on $\Mza$.
So there exists only finitely many positive lower bounds of 
$\varphi_{\cA *}(\Delta_{\cA})\cdot B$ with respect to each stratum.
After taking the minimum, we get the global positive lower bound 
$\epsilon_{\cA}$ of intersection numbers.
\end{proof}

\begin{remark}\label{rmk-localglobal}
Theorem \ref{thm-mainthm} shows an unexpected duality.
Let $(C, x_{1}, x_{2}, \cdots, x_{n})$ be a stable pointed rational curve.
Then the log canonical model $C(\omega_{C}+\sum a_{i}x_{i})$ is 
an $\cA$-stable curve. More precisely, it is 
$\varphi_{\cA}(C, x_{1}, x_{2}, \cdots, x_{n})$. 
The same weight datum determines the log canonical model 
$\Mzn(K_{\Mzn}+\sum a_{i}\psi_{i})$.
\end{remark}

\begin{remark}\label{rmk-symmetriccase}
Suppose that the weight datum $\cA = (a_{1},\cdots, a_{n})$ is symmetric, i.e,
$a_{1} = \cdots = a_{n} = \alpha$ for some $2/n < \alpha \le 1$.
Then by \cite[Lemma 1]{FarkasGibney} and 
\cite[Proposition 2]{Pandharipande},
\begin{equation}
	\psi = \sum_{j=2}^{\lfloor n/2 \rfloor} \frac{j(n-j)}{n-1} D_{j}
	= K_{\Mzn}+2D.
\end{equation}
So for $\alpha > 0$,
\begin{equation}
	K_{\Mzn}+\alpha \psi = (1+\alpha)(K_{\Mzn}+
	\frac{2\alpha}{1+\alpha} D).
\end{equation}
Therefore the log canonical model of the pair $(\Mzn, K_{\Mzn}+\alpha \psi)$
is equal to the log canonical model of the pair $(\Mzn, K_{\Mzn}+
\frac{2\alpha}{1+\alpha}D)$.
If we substitute $\beta = \frac{2\alpha}{1+\alpha}$, then we 
get Theorem \ref{thm-symmetriccase}.
Hence it is a generalization of Simpson's theorem 
(Theorem \ref{thm-symmetriccase}).
\end{remark}

\begin{remark}\label{rmk-FarkasGibneyresult}
A divisor $\Delta$ on $\Mzn$ is called \emph{log canonical} if 
\begin{equation}\label{eqn-logcanonicaldivisor}
	\Delta \equiv 
	r(K_{\Mzn}+\sum_{i \subset [n], 2 \le |I| \le n-2} c_{I}D_{I})
\end{equation}
for some $r > 0$ and $0 \le c_{I} \le 1$
\cite[Definition 6.2]{AlexeevGibneySwinarski}.
For a log canonical divisor $\Delta$, it is nef if and only if $\Delta$ intersects 
with vital curves nonnegatively (\cite[Theorem 5]{FarkasGibney}).
So if $\varphi_{\cA}^{*}\varphi_{\cA *}(\Delta_{\cA})$ in 
\eqref{eqn-pushforwardpullback} is log canonical, then 
the proof in \cite{Moon1} is sufficient.
But for some weight data, $\varphi_{\cA}^{*}\varphi_{\cA *}(\Delta_{\cA})$
is not log canonical. For example, if $n=10$ and
$\cA = (1, 1, \epsilon, \epsilon, \cdots, \epsilon)$ for sufficiently 
small $\epsilon > 0$, then by using computer algebra system, 
we can check that $\varphi_{\cA}^{*}\varphi_{\cA *}(\Delta_{\cA})$
is not log canonical.
So for a complete proof, we need the argument in this paper.
\end{remark}



\bibliographystyle{alpha}

\end{document}